\documentclass[11pt]{amsart}

\usepackage{amssymb, amsmath}
\usepackage{diagbox}

\usepackage{tikz, tikz-cd}
\usepackage{enumitem}
\usepackage{mathtools}
\usepackage{todonotes}
\usepackage{amsthm}
\usepackage{amsfonts}
\usepackage{amsmath}
\usepackage[all]{xy}

\theoremstyle{plain}
 \newtheorem{thm}{Theorem}[section]
 
 \newtheorem{cor}[thm]{Corollary}
\newtheorem{lem}[thm]{Lemma}
\newtheorem{prop}[thm]{Proposition}

\theoremstyle{definition}

\theoremstyle{plain}
\newtheorem{theorem}[thm]{Theorem}

\newtheorem{lemma}[thm]{Lemma}

\theoremstyle{definition}

\newtheorem{defin}[thm]{Definition}
\newtheorem{example}[thm]{Example}

\newtheorem{remark}[thm]{Remark}

\numberwithin{equation}{section}

\newcommand{\sA}{{\mathcal A}}

\newcommand{\sH}{{\mathcal H}}

\newcommand{\sK}{{\mathcal K}}

\newcommand{\sP}{{\mathcal P}}

\newcommand{\B}{{\mathbb B}}

\newcommand{\E}{{\mathbb E}}

\newcommand{\N}{{\mathbb N}}

\newcommand{\R}{{\mathbb R}}

\newcommand{\hol}{\ensuremath{\mathcal{O}}}

\newcommand\la{\lambda}

\newcommand\e{\epsilon}
\newcommand\al{\alpha}

\newcommand\Ga{\Gamma}

\newcommand\ga{\gamma}
\newcommand\de{\delta}

\DeclareMathOperator{\Pic}{Pic}

\newcommand{\HH}{\ensuremath{\mathbb{H}}}

\newcommand{\FF}{\ensuremath{\mathbb{F}}}

\newcommand{\ra}{\ensuremath{\rightarrow}}

\newcommand{\CC}{\mathbb{C}}

\newcommand{\PP}{\mathbb{P}}
\newcommand{\QQ}{\mathbb{Q}}
\newcommand{\RR}{\mathbb{R}}
\newcommand{\ZZ}{\mathbb{Z}}

\newcommand{\Num}{\mathrm{Num}}

\begin{document}
\title[
Mori dream spaces and $\QQ$-homology quadrics]{
 Mori dream spaces and  $\QQ$-homology quadrics}
 \author{Paolo Cascini}
\author{Fabrizio Catanese}
\author{Yifan Chen}
\author{JongHae Keum}
\date{\today}

\address{Paolo Cascini,
Department of Mathematics\\
Imperial College London\\
180 Queen's Gate\\
London SW7 2AZ, UK}
\email{p.cascini@imperial.ac.uk}

\address{Fabrizio Catanese,
 Mathematisches Institut der Universit\"{a}t
Bayreuth, NW II\\ Universit\"{a}tsstr. 30,
95447 Bayreuth, Germany.}
\address{
Fabrizio Catanese,  Korea Institute for Advanced Study, Hoegiro 87, Seoul,
133--722.}
\email{Fabrizio.Catanese@uni-bayreuth.de}

\address{Yifan Chen, School of Mathematical Sciences, Beihang University, 9 Nanshan Street,
Shahe Higher Education Park, Changping, Beijing, 102206, P. R. China}
\email{chenyifan1984@buaa.edu.cn}

\address{JongHae Keum, HCMC, Korea Institute for Advanced Study, Seoul 02455, Republic of Korea} \email{jhkeum@kias.re.kr}


\maketitle

\begin{abstract}
We show that Shavel type surfaces are fake quadrics of even type which are not Mori dream surfaces,
yet there are infinitely many primes $p$ such that the reduction modulo $p$ is a Mori dream surface.

We investigate fake quadrics, first concerning  the property of being Mori dream surfaces,
then we try to determine which surfaces isogenous to a product are fake quadrics of even type.

\end{abstract}

\section{Introduction}

In the paper \cite{moridream} the authors considered  complex surfaces of general type
with $q=p_g=0$ which are Mori dream surfaces, and asked in section 3.2 whether there are fake quadrics which
are not Mori dream surfaces.

Recall that, for complex surfaces isogenous to a product of curves, it was established  in  \cite{moridream} and \cite{f-l}
 that they are Mori dream surfaces.

At the Hefei Conference in September 2024, the second author pointed out that such examples
should be provided  by the surfaces constructed by Shavel in 1978 \cite{shavel}.

The first aim of this short note is therefore to give a complete proof of this assertion.

The second aim is to discuss several  problems related to  minimal surfaces  of general type
which are $\QQ$-homology quadrics: that is,  surfaces with  $q=p_g=0$, and with  second Betti number $b_2(S)=2$
(equivalently, with $K^2_S=8$).

 Up to now, all known such examples have universal cover equal to the bidisk $\HH \times\HH$.

 Indeed Hirzebruch (\cite{hirz} Problem 25,  see also pages 779-780 of \cite{ges-werke}) was the first to ask the question whether there exists a surface of general type which is homeomorphic to $\PP^1 \times \PP^1$,  and one can respectively ask the same question for a surface
  homeomorphic to the blow up $\FF_1$ of $\PP^2$ in one point. The answer is
 suspected to be negative.

 The above two manifolds are topologically distinguished by the property that the intersection form on the Severi group $Num(S)$ is even in the first case,
 odd in the second.

We show here that many of our surfaces (surfaces  of general type
which are $\QQ$-homology quadrics) have even intersection form.    The existence of the  case
of odd intersection form  is still unknown.

 Moreover, the third question which we consider is: what happens when, instead of complex surfaces,
 we consider surfaces defined over an algebraically closed field of positive characteristic? When
  are they Mori dream surfaces?

   This is related to a beautiful conjecture by Ekedahl, Shepherd -Barron and Taylor \cite{EST}  about
   algebraic  integrability of
   foliations via reduction modulo primes $p$.

   We can summarize  our result in this regard as follows

   \begin{theorem}\label{shavel}
   A Shavel type surface $S$  is an even  fake quadric which is not a Mori dream surface.

   There are infinitely many primes $p$ such that the reduction of $S$ modulo $p$ is a Mori dream surface.

   \end{theorem}

  We also give results stating when a fake quadric over $\QQ$ is a Mori dream space.

   {\bf Acknowledgement.} Our collaboration started at the Hefei USTC Conference and continued at the Shanghai Fudan
   Workshop on birational geometry, held in September 2024; we are very grateful to Meng Chen and Lei Zhang for organizing these events.

   The question about the  existence of fake quadrics of odd type was raised by Jianqiang Yang in email correspondence with
   the second author (in December 2022).


 The first author is partially supported by a Simons Collaboration grant.
The last author is supported by National Research
Foundation of Korea (NRF 2022R1A2B5B03002625).

\section{Definitions and basic properties}

In these first sections we shall mostly work with projective smooth surfaces defined over the field $\CC$,
most definitions however make also sense if $\CC$ is replaced by an algebraically closed field $\sK$ of arbitrary characteristic.

\begin{defin}[ $\QQ$-homology Quadric]\label{defn:rationalquadric}
Let $S$ be a smooth projective surface over $\CC$.

The surface $S$ is called a   {\bf $\QQ$-homology quadric} if
$q(S)=p_g(S)=0$, $b_2(S)=2$.

In turn, it will be called an {\bf even homology  quadric} if
\begin{enumerate}
    \item $q(S)=p_g(S)=0$, $b_2(S)=2$;
    \item the intersection form on $\mathrm{Num}(S)$ is even, that is,
    \begin{align*}
        \left(\begin{array}{cc}
               0 & 1\\
               1 & 0
              \end{array}\right).
    \end{align*}
\end{enumerate}

We shall call it an {\bf odd homology quadric } if instead the intersection form is odd (hence diagonalizable with diagonal entries $ (+1, -1)$).
\end{defin}
For a $\QQ$-homology quadric $S$, condition (1) and the Noether formula imply
$$\chi(\hol_S)=1, c_2(S)=4, K_S^2=8.$$
Also the long exact sequence of cohomology groups associated to the exponential exact sequence shows that
 $$c_1 \colon \mathrm{Pic}(S) \rightarrow H^2(S,\ZZ)$$
 is an isomorphism. We identify these two groups and denote by
 $\mathrm{Tors}(S)$ their torsion subgroup. Then
 $$\mathrm{Num}(S)\cong \mathrm{Pic}(S)/\mathrm{Tors}(S).$$

 In the case of an even homology   quadric, condition (2) implies that $S$ is minimal. Hence, by surface classification, for an even homology quadric
\begin{itemize}
    \item either $S$ is rational, and then $S\cong \mathbb{F}_{2n}$ for $n\ge 0$;
    \item or $S$ is a minimal surface of general type.
\end{itemize}
In the former case $S$ is simply connected, $\mathrm{Tors}(S)=0$ and $K_S$ is not ample.

 In the case of an odd homology quadric
\begin{itemize}
    \item either $S$ is rational, and then $S\cong \mathbb{F}_{2n+1}$ for $n\ge 0$;
    \item or $S$ is a, non necessarily minimal, surface of general type.
\end{itemize}

\begin{defin}[Even fake Quadric]\label{defn:fakequadric} Let $S$ be a smooth projective surface over $\CC$.
The surface $S$ will be  called \footnote{The reader should be made aware that, in the literature, sometimes $\QQ$-homology quadrics
of general type  are referred to as fake quadrics, without further specification whether the intersection form is even or odd, see \cite{ges-werke}, \cite{dzambic}, \cite{d-r}, \cite{f-l}.} an {\bf even fake quadric} if it is an even  homology quadric and is of general type.
\end{defin}

In general, a minimal smooth complex projective surface of general type with $p_g=q=0, K^2=8$ has Picard number $2$
and $K_S$ is ample by \cite[Proposition~2.1.1]{miyaoka}.
Also $\mathrm{Tors}(S)$ can  be nonzero (see \cite{bcf}).

For $L \in \mathrm{Pic}(S)$,
we denote by $[L]$ its class in $\mathrm{Num}(S)$.

\begin{defin}[Odd fake Quadric]\label{defn:fakeF1}Let $S$ be a smooth projective surface over $\CC$.
  If $S$  is an odd homology quadric which is  of general type,
then either $S$ is not minimal and it is a one point blow up of a fake projective plane, or
we call the  surface $S$ an {\bf odd fake quadric}, which means:
\begin{enumerate}
    \item $S$ is minimal of general type with $K_S^2=8, p_g(S)=0$;
    \item the intersection form on $\mathrm{Num}(S)$ is odd.
\end{enumerate}
\end{defin}

Therefore smooth minimal surfaces of general type with $K^2=8$ and $p_g=0$ are divided into two classes:
the even and the odd fake quadrics.

\begin{lem}\label{fibration}Let $S$ be a fake quadric.

Assume that $f \colon S \rightarrow \PP^1$ is a fibration and  denote by $F$ the general fibre of $f$. Then
\begin{enumerate}
    \item  any fibre $F_t = f^*(t)$ has irreducible support;
    \item the multiplicity $m_t$ of any multiple fibre $F_t = m_t F'_t$  divides $g-1$, where $ g$ is the genus of $F$; and
    \item  the class of the   fibre $[F]$  is divisible by $d$ in $Num(S)$, where $d$ is the  least common multiple of the exponents
    $m_i$ of the multiple fibres.
\end{enumerate}
\end{lem}
\proof (1)  follows from Zariski's lemma (the intersection form on the fibre components is seminegative with nullity $1$), and from the fact that the Picard number  $\rho(S): = rank (Num(S))=2$.

(2) follows by adjunction.

(3)
If $d$ is the least common multiple of the multiplicities $m_1, \dots, m_s$
 of the  multiple fibres, then we may write $\frac{1}{m_i} = \frac{ r_i} {d} $, where the $r_i$'s  have GCD equal to $1$. Thus,  we can write
 $ \frac{1}{d} $ as a sum of the rational numbers $ \frac{1}{m_i} $, hence proving that  $\frac{1}{d} [F] \in Num(S)$, since it is an integer
 linear combination of the classes $F'_i$.

 Therefore,  the divisibility  of $[F]$  is a multiple of  $d$.
 \qed

\begin{lem}Keep the assumption of the previous lemma \ref{fibration}.

Assume that $S$ has another fibration $f'\colon S\to \mathbb P^1$ with general fibre $F'$. Then (we use the classical symbol $ \sim : = \sim_{num}$
for numerical equivalence)
\begin{enumerate}
    \item $FF'=(g-1)(g'-1)$;
    \item $K_S\sim \frac{2}{g-1}F+\frac{2}{g'-1}F'$; and
    \item the intersection form of $S$ is even if $\frac{1}{g-1}F, \frac{1}{g'-1}F' \in Num(S)$.
\end{enumerate}
\end{lem}
\proof The matrix of intersection numbers of $K_S, F, F'$ has determinant
$$8(g-1)(g'-1)x-8x^2,~~x=F   F',$$
which must be $0$ since $\rho(S)=2$, hence (1) and (2) follow right away.

For (3), these classes generate a unimodular lattice, hence all of $Num(S)$.
\qed

\begin{lem}\label{evenodd}Let $S$ be a fake quadric. Then
 \footnote{ by Wu's formula, saying that $K_S$ induces the second Stiefel Whitney class $w_2(S) \in H^2 (S, \ZZ/2)$,  and by the universal coefficients formula, $S$ is spin (i.e., $w_2(S) =0$) if and only if $K_S$ is divisible by $2$ in $H^2 (S, \ZZ)$, a fact which is often expressed by saying that $S$ is {\bf even}; being an even surface is  a stronger notion than requiring that the intersection form is even, as one can see from  the example  of  Enriques surfaces.}  the intersection form is even if and only if $K_S$ is divisible by 2 in  $\mathrm{Num}(S)$.
\end{lem}
\proof $K_S   D \equiv D^2 (mod \  2)$, and the intersection form on $$Num(S) = H^2(S, \ZZ) / Tors(S)$$ is unimodular.
\qed

\section{Surfaces isogenous to a product}

\begin{defin}Let $S$ be a smooth projective surface.
The surface $S$ is said to be {\bf isogenous to a  higher product}  \cite{isogenous}  if
$$S\cong (C_1 \times C_2)/G,$$
where $C_i$ is a smooth curve with $g(C_i)\ge 2$ and $G$ is a finite group acting faithfully and freely on
$C_1\times C_2$.

 If there are respective actions of $G$ on $C_1$ and $C_2$ such that  $G$ acts
 by the diagonal action $ g(x,y) = (gx, gy)$  on $C_1\times C_2$, then $S$ is called of {\bf unmixed type}.

If some element of $G$ exchanges two factors, then $S$ is called of {\bf mixed type}.
\end{defin}
\noindent~\textbf{Convention:}~In this article, we shall only consider  surfaces isogenous to a product with $p_g(S)=0$.

\bigskip

 The next question which we ask is to determine which of the $\QQ$-homology quadrics which are isogenous to a
higher product are even or odd fake quadrics.

\begin{example}
The classical Beauville surface is an even  fake quadric.

\end{example}
\begin{proof}
Here $C_1= C_2$ are the Fermat quintics in $\PP^2$, and the group $G = \mu_5 \times \mu_5$
acts by the linear action
$$ (\zeta_1, \zeta_2) (x_0, x_1, x_2) =  (x_0, \zeta_1 x_1, \zeta_2 x_2) .$$
Hence $\hol_C(1) $ is a $G$-linearized bundle with square $K_C$,
see also (\cite{cat}).

Thus,  $\hol_{C_1}(1) \otimes \hol_{C_2} (1) $ is a $G$-linearized bundle which
descends to a line bundle $L$ with square  $K_S$.

Alternatively, the fibres of both projections are divisibile by 5 and yield $F'_1, F'_2$
such that $F'_1   F'_2 =1$ and $(F'_j)^2=0$.
\end{proof}

We can give a partial  answer to the question for surfaces isogenous to a product of unmixed type with $p_g=0$,
which have been classified
in \cite{product} (their torsion groups have been classified in \cite{bcf}).

\begin{theorem}
Let $S = (C_1\times C_2)/G$ be a surface isogenous to a product of unmixed type,
with $p_g(S) = 0$; then $G$ is one of the groups in the Table \ref{tab1} and the mutiplicities of the multiple fibres for
 the two natural fibrations $ S \ra C_i/G$ are listed in the table.
 For each case in the list we have an irreducible component of the moduli space, whose dimension is denoted by
 $D$. The property of being an even, respectively an odd  homology quadric
   and the first homology group of $S$ are given in the third last  column, respectively in the last column.

\begin{table}[!h]
\begin{tabular}{c|c|c|c|c|c|c}
$G$ & $Id(G)$ & $T_1$ & $T_2$ & $parity$&$D$&  $H_1(S,\mathbb Z)$\\
\hline
$\mathcal{A}_5$ & $\langle 60,5\rangle$& $[2,5,5]$& $[3,3,3,3]$&?&1&$(\mathbb Z_3)^2\times (\mathbb Z_{15})$\\
$\mathcal{A}_5$ & $\langle 60,5\rangle$& $[5,5,5]$& $[2,2,2,3]$  &?&1&$(\mathbb Z_{10})^2$\\
$\mathcal{A}_5$ & $\langle 60,5\rangle$& $[3,3,5]$& $[2,2,2,2,2]$  &? &2&$(\mathbb Z_2)^3\times \mathbb Z_6$\\
$\mathcal S_ 4 \times \mathbb Z_2$& $\langle 48,48 \rangle$& $[2,4,6]$& $[2,2,2,2,2,2]$  &?&3&
$(\mathbb Z_2)^4\times \mathbb Z_4$\\
G(32) & $\langle 32,27 \rangle$& $[2,2,4,4]$& $[2,2,2,4]$ &?&2&$(\mathbb Z_2)^2\times \mathbb Z_4\times \mathbb Z_8$\\
$(\mathbb Z_5)^2$ & $\langle 25,2\rangle$& $[5,5,5]$& $[5,5,5]$ &even&0&$(\mathbb Z_5)^3$\\
$\mathcal{S}_4$ & $\langle 24,12\rangle$& $[3,4,4]$& $[2,2,2,2,2,2]$ &even&3&$(\mathbb Z_2)^4\times \mathbb Z_8$\\
G(16) & $\langle 16,3\rangle$& $[2,2,4,4]$& $[2,2,4,4]$&even&2&$(\mathbb Z_2)^2\times \mathbb Z_4\times \mathbb Z_8$\\
$D_4\times \mathbb Z_2$ & $\langle 16,11\rangle$& $[2,2,2,4]$& $[2,2,2,2,2,2]$ &?&4&$(\mathbb Z_2)^3\times (\mathbb Z_4)^2$\\
$(\mathbb Z_2)^4 $ & $\langle 16,14\rangle$& $[2,2,2,2,2]$&$[2,2,2,2,2]$  &even &4&$(\mathbb Z_4)^4$\\
$(\mathbb Z_3)^2$ & $\langle 9,2\rangle$& $[3,3,3,3]$& $[3,3,3,3]$&even&2&$(\mathbb Z_3)^5$\\
$(\mathbb Z_2)^3$ & $\langle 8,5\rangle$& $[2,2,2,2,2]$& $[2,2,2,2,2,2]$  &?&5&$(\mathbb Z_2)^4\times (\mathbb Z_4)^2$\\
\end{tabular}	
\caption{ }\label{tab1}
\end{table}
\end{theorem}

\begin{proof}

In view of the cited results in \cite{product} and \cite{bcf}, it suffices to prove the assertion about the intersection form
being even, respectively odd.

We  use now (3) of Lemma \ref{fibration} showing that $F_j = d_j \Phi_j$, where $d_j$ is the least common multiple of the multiplicities
of the fibration with general fibre $F_j$.

Since $F_1   F_2 = |G|$, we conclude that $$\Phi_1   \Phi_2 =\frac{ |G| }{d_1 d_2}.$$
If $d_1 d_2 = |G|$, then $\Phi_1   \Phi_2 =1$, and, since $\Phi_j^2 =0$, we have an even intersection form.

Inspecting the list, we see that $\Phi_1   \Phi_2 \in \{ 1, 2,4\}$.

If we have an even intersection form, then we have a basis $e_1, e_2$ of the lattice $Num(S)$ with $e_j^2=0$,
and $e_1   e_2 = 1$.

Hence wlog we may assume that $[\Phi_j ]= a_j e_j$, and therefore $\Phi_1   \Phi_2 =a_1 a_2$.

If $2 | a_1 a_2$, then, for some $j$,  $2 | a_j$, and $\Phi_j$ is further divisible by $2$.

We could conclude that the intersection form is odd, in case $\Phi_1   \Phi_2 = 2$, if we knew that each $\Phi_j$
is not divisible by $2$.

In fact, if  the intersection form is odd, then we have a basis $q_1, q_2$ of the lattice $Num(S)$ with $q_1^2=1$, $q_2^2=-1$,
and $q_1   q_2 = 0$.

Then $\Phi_1, \Phi_2$ must be multiples of $q_1+ q_2$, respectively $q_1- q_2$, and indeed $(q_1+ q_2)
  (q_1- q_2) = 2$.

In the only case (with group ($\ZZ/2)^4$) where $\Phi_1   \Phi_2 = 4$, the divisibility index of $\Phi_1$ equals the one of $\Phi_2$ by the
symmetry of the roles of the two curves $C_1, C_2,$ including the associated monodromies. Hence either $\Phi_1$ and $\Phi_2$ are both 2-divisible,
and the intersection form is even, or the intersection form is odd and  $\Phi_1$ and $\Phi_2$ are both indivisible.
But then $\{ \Phi_1, \Phi_2\} = \{ (q_1+ q_2) , (q_1- q_2) \}$ and $(q_1+ q_2)
  (q_1- q_2) = 2$ contradicts  $\Phi_1   \Phi_2 = 4$.
  
\end{proof}

\begin{remark}

In the case where  $\Phi_1   \Phi_2 = 2$, it is easy to see which of the two divisors may be 2-divisible.

In fact, using the ramification formula for $ S \ra (C_1/G) \times (C_2/G) = \PP^1  \times  \PP^1 $,
we see that, setting $m_1, \dots, m_s$ to be the multiplicities of the multiple fibres in the first fibration,
and $ n_1, \dots, n_{s'}$ those in the second, then
$$ K_S = (-2 + \sum_j (1 - \frac{1}{m_j} )) F_1 +  (-2 + \sum_i (1 - \frac{1}{n_i} ) ) F_2 $$
and, in $Num(S) / 2 Num(S)$, we have
$$ K_S  \equiv  \sum_j( d_1 - r_j) \Phi_1 + \sum_i( d_2 - r'_i) \Phi_2 \equiv \de_1 \Phi_1 + \de_2 \Phi_2, \ \de_1, \de_2 \in \{0,1\} $$

We see by direct inspection  that exactly one $\de_j$ equals  $1$, the other is $0$.

Then $[K_S] \in Num(S)$ is 2-divisible if and only if the $[\Phi_j]$ with $\de_j =1$ is 2-divisible.

For instance in the last case we have $\de_1=1$, $\de_2=0$, indeed $K_S \sim \Phi_1 + 2 \Phi_2$.

\end{remark}

\begin{remark}
\begin{enumerate}
\item In the last case we can prove that $K_S$ is not 2-divisible, since there is no $G$-linearized thetacharacteristic
on the curve $C_1$, a hyperelliptic curve of genus $g_1=3$. Indeed, the only $G$-fixed thetacheracteristics are the hyperelliptic
divisor $\sH$, which does not admit a $G$-linearization, and $ P_1 + P_2 + P_3 + P_4 -  \sH$,
where the $P_j$'s are Weierstrass points and their sum is a $G$-orbit (hence $ \hol_{C_1} (  P_1 + P_2 + P_3 + P_4)$
admits a $G$-linearization).

\item  On the other hand, showing that $[K_S]$ is not 2-divisible is harder, in view
of the existence of torsion divisors of order $4$.

\item Our observation (1) shows that, given a Fuchsian group $\Ga < \PP SL (2, \RR)$ which is not torsion free,
the embedding $\Ga \hookrightarrow  \PP SL (2, \RR)$ does not necessarily  lift to $  SL (2, \RR)$
(unlike the case where $\Ga$ is cocompact and torsion free).
\end{enumerate}
\end{remark}

\section{Even fake quadrics}

\noindent{\textbf{Assumption}:}
In this section, we let $S$ be an even  fake quadric.
Recall that, over the complex number field $\CC$,  $S$ contains no smooth rational curves \cite[Proposition~2.1.1]{miyaoka}
and in particular $K_S$ is ample.

\subsection{Nef cone of an even  fake quadric}
\begin{lem}\label{lem:numS}
\begin{enumerate}
\item There exist $L_1, L_2 \in \Pic(S)$ such that $L_1  L_2=1$, $K_SL_i=2,
 L_i^2=0$ for $i=1,2$.
 \item  For any $L_1, L_2$ as in (1),
$$\mathrm{Num}(S)=\ZZ[L_1]\oplus \ZZ[L_2].$$
\item
 $~ K_S \sim 2L_1+2L_2$.
\end{enumerate}
\end{lem}

 The condition that the universal covering of $S$ is the bidisk can be formulated as follows: there exists a 2-torsion divisor $\eta$
such that $$H^0( S^2(\Omega^1_S)(-K_S + \eta)) \neq 0,$$ see for instance \cite{cf}, \cite{cds}; this condition is equivalent  to the splitting of the cotangent bundle on a suitable unramified double covering of $S$.

From now on, we fix $L_1, L_2 \in \Pic(S)$ as in \ref{lem:numS}~(1).

\begin{prop}\label{prop:nefdivisor}
\begin{enumerate}
\item Any effective divisor on $S$ is nef.
\item Any nef and big divisor on $S$ is ample.
\item $L_i$ is nef for $i=1, 2$.
\item For any strictly effective divisor $D$ on $S$, if $D^2=0$, then $D\sim aL_1$ or $D\sim a L_2$ for some $a \in \ZZ_{>0}$.
\end{enumerate}
\end{prop}

\begin{remark}For (4), we do not claim that such $D$ exists.
\end{remark}

\proof

For (1), it suffices to show for any irreducible curve $C$, we have $C^2\ge 0$.
Assume by contradiction that $C^2<0$. Assume $C\sim a L_1+b L_2$ for $a, b\in \ZZ$. Then
$$C^2=2ab,~~ K_S   C=2a+2b.$$
We may assume that $a >0$ and $b<0$. By the adjunction formula,
$$-2\le 2p_a(C)-2=C^2+K_SC=2(a+b+ab)=  2a (1+b) + 2 b.$$
Therefore $b=-1$, $C \cong \PP^1$, and this    contradicts the fact that $S$ contains no smooth rational curve.

For (2), let $D$ be a nef and big divisor on $S$. Then $D^2>0$. Let $C$ be any irreducible curve.
Since $C$ is nef, $DC\ge 0$. Also, since $C^2\ge 0$, $DC>0$ by the algebraic index theorem. Therefore
$D$ is ample.

For (3) and (4), let $D$ be an effective divisor. Assume that $D\sim aL_1+bL_2$. Then $D^2=2ab$ and $K_SD=2(a+b)$. Since $D$ is nef and $K_S$ is ample, $a\ge 0, b\ge 0, a+b>0$. Then
$L_1D=b\ge 0, L_2D=a\ge 0$.  And if $D^2=0$, then $a=0$ or $b=0$.
\qed

\begin{cor}\label{cor:cones} Let $L_1, L_2$ as in \ref{lem:numS}~(1).
In $\mathrm{Num}(S)_\RR$,
\begin{align*}
    \mathrm{Amp}(S)&=\{ a[L_1]+b[L_2]~|~ a, b \in \RR_+\},\\
    \mathrm{Nef}(S)&=\{ a[L_1]+b[L_2]~|~a, b\in \RR_{\ge 0}\},\\
    \mathrm{Nef}(S)&=\overline{\mathrm{Eff}(S)}.
\end{align*}
\end{cor}

\begin{lem}\label{lem:cotangent}Assume that the cotangent sheaf of $S$ splits as the direct sum of two invertible sheaves:
       $$\Omega_{S}^1=\hol_S(A_1)\oplus \hol_S(A_2).$$
       Then either $[A_1]=2[L_1], [A_2]=2[L_2]$ or $[A_1]=2[L_2], [A_2]=2[L_1]$.

        Moreover the universal covering of $S$ is the bidisk $\HH \times \HH$, and $S = \HH \times \HH/\Ga$,
       where $ \Ga < \PP SL (2, \RR) \times \PP SL (2, \RR) $.
\end{lem}
\proof Assume that $[A_1]=a[L_1]+b[L_2]$ with $a, b \in \ZZ$. Since $K_S=A_1+A_2$ and $[K_S]=2[L_1]+2[L_2]$,
$[A_2]=(2-a)[L_1]+(2-b)[L_2]$.

A Chern class computation shows that  $A_1A_2=c_2(S)=4$. That is,
$$a(2-b)+b(2-a)=4,~~i.e.,~~(a-1)(b-1)=-1.$$
Therefore either $a=2, b=0$ or $a=0, b=2$.

The last assertion  is known since long time, see   \cite{yau}, \cite{beauville}.

\qed.

\subsection{Even fake quadrics and Mori dream surfaces}

The following theorem follows from Theorems 3.9 and 3.10 of   \cite{keum-lee-2},
but we give another proof for the reader' s  convenience.

\begin{thm}\label{thm:moridream}
Let $S$ be an even  fake quadric.
Then $S$ is a Mori dream surface if and only if $L_1$ and $L_2$ are semiample,  equivalently,  if and only if
$S$ admits a finite morphism to $\PP^1 \times \PP^1$.

 The same assertion is true for any fake quadric which does not contain a negative curve.
\end{thm}

\proof Note that $\Pic(S)\cong H^2(S,\ZZ)$ is finitely generated.
According to \cite[Corollary 2.6]{K3},
$S$ is a Mori dream surface if and only if $\mathrm{Eff}(S)$ is rational polyhedral and $\mathrm{Nef}(S)=\mathrm{SAmp}(S)$.

By \ref{cor:cones},
$$\mathrm{Nef}(S)=\overline{\mathrm{Eff}}(S)=\{a[L_1]+b[L_2]~|~a, b \in \RR_{\ge 0}\}.$$
Since $\mathrm{Eff}(S) \supseteq \mathrm{SAmp}(S)$,
it follows that $S$ is Mori dream surface if and only if $\mathrm{Nef}(S)=\mathrm{SAmp}(S)$,
if and only if $L_1, L_2$ are semi-ample.

It follows that  $S$ is a Mori dream surface if and only if $S$ has two fibrations $f_1, f_2 : S \ra \PP^1$.
These combine to yield a morphism $ f : S \ra \PP^1 \times \PP^1$ which is necessarily finite since the
second Betti  number of $S$ equals $2$.

Conversely, if $S$ has a finite morphism  $ f : S \ra \PP^1 \times \PP^1$,  $ q(S)=0$, and the
second Betti  number of $S$ equals $2$, then $S$ is a Mori dream surface with $ q(S) = p_g (S)=0$,
hence in particular it  is a $\QQ$-homology quadric.

Since the property of being a Mori dream space depends on the structure of $Num(S) \otimes \RR$,
by the cited criteria, in the case of an odd quadric we take a
basis of $Num(S)$ as in Lemma \ref{lem:intersectionF1} and set $L_1 : = Q_1 + Q_2$, $L_1 : = Q_1 - Q_2$.

If there are no negative curves, then  the cones $\mathrm{Nef}(S)$ and the closure of $\mathrm{Eff}(S)$ are again
equal to the first quadrant, and the proof runs exactly as in the even case.
\qed

\section{The Shavel type surfaces}

\begin{defin}A smooth projective surface $S$ shall be  called a {\bf Shavel surface of special unmixed type} if
 $$p_g(S)=q(S)=0, S=\mathbb{H}^2/\Gamma$$
 where $\Gamma$ is a cocompact, torsion-free, irreducible subgroup of
 $$\mathrm{Aut}(\mathbb{H}^2) \simeq \PP \mathrm{SL}(2, \RR)^2 \rtimes \ZZ/2,$$
 such that
  $$\Gamma <  \mathrm{SL}(2,\RR)\times \mathrm{SL}(2,\RR).$$

\end{defin}

Note that for a Shavel surface $S$ of unmixed type,
for $\forall \gamma=(\gamma_1, \gamma_2) \in \Gamma$, $\forall z=(z_1, z_2) \in \mathbb{H}^2$,
$$\gamma z=(\gamma_1 z_1, \gamma_2 z_2).$$
Hence $S$ admits two smooth foliations and $\Omega^1_S$ splits as the direct sum of  two invertible sheaves:
 $$\Omega^1_S=\L_1\oplus \L_2.$$

\begin{prop}\label{prop:shaveleven}A Shavel surface $S$ of special unmixed type is an even  fake quadric and $K_S$ is divisibe by 2 in $Pic(S)$.
\end{prop}

\begin{proof} It suffices to show $K_S$ is divisible by $2$ in $Pic(S)$.

The automorphic factor of the canonical bundle  is  the inverse of the jacobian determinant $\frac{1}{(c_1z_1+d_1)^2(c_2z_2+d_2)^2}$,
for $\gamma=(\gamma_1, \gamma_2)$ and where
\begin{align*}
\gamma_i=\left(\begin{array}{cc} a_i & b_i \\ c_i & d_i\end{array}\right), i=1, 2.
\end{align*}

This shows immediately that  $K_S$ is the square of the automorphic factor
$(c_1z_1+d_1)(c_2z_2+d_2)$, hence  
$$ K_S =  2 (L_1 + L_2), \L_j =  2 L_j$$
 and our claim follows.
\end{proof}

\begin{thm}\label{thm:shavel}A Shavel surface of unmixed type $S$ is not a Mori-dream surface.
\end{thm}
\proof We have seen that $S$ is an even  fake quadric in  the above proposition \ref{prop:shaveleven}. We use the notation of  Section~2.

It suffices to prove that $|nL_1|=\emptyset$ for any $n\ge 1$.

 As remarked above,
$S$ admits two smooth foliations and $\Omega^1_S$ splits as the sum of  two invertible sheaves:
 $$\Omega^1_S=\L_1\oplus \L_2,$$
 where, by \ref{prop:shaveleven}, we have $\L_1= 2L_1, \L_2 = 2L_2$.

Since $\mathrm{Tors}(S)$ is a finite abelian group, it suffices to show that
$$| n  L_1 |=\emptyset, \forall n\ge 1.$$

In fact, $n L_1$ is an automorphic line bundle on $\mathbb{H}\times \mathbb{H}$
corresponding to the   following cocycle: first we take the  
first projection map $$p_1 \colon \PP \mathrm{SL}(2,\RR) \times \PP\mathrm{SL}(2,\RR) \rightarrow \PP \mathrm{SL}(2,\RR)$$
$$\gamma=(\gamma_1, \gamma_2) \mapsto \gamma_1.$$

 Then to $\ga_1$,  such that 
 $$\ga_1(z_1) = \frac{a (\ga_1)  z_1 + b(\ga_1)}{c (\ga_1)  z_1 + d(\ga_1)}$$
 we associate the automorphic factor  $(c (\ga_1)  z_1 + d(\ga_1) )$.
 
Since $\Gamma$ is irreducible, $p_1(\Gamma)$ is dense.

We claim that $H^0(S, n L_1)=0$ for $ n \geq 1$.

In fact, every section of $H^0(S, n L_1)$ is represented by a function $f$ which satisfies the functional equation
$$ f (\ga_1 z_1, \ga_2 z_2) = (c (\ga_1)  z_1 + d(\ga_1) )^n f (z_1, z_2).$$

By density of $p_1(\Gamma)$, this holds for each $\ga_1 \in \mathrm{SL}(2, \RR)$.

Here $\ga_2$ is obtained by applying the involution of the quaternion algebra. If we take now $\ga_1$ to be in a maximal compact subgroup,  the stabilizer of one point, then the same holds for $\ga_2$, hence using the biholomorphism of $\HH$ with the unit disk, and choosing suitable coordinates, we can assume that
$\ga_1 (z_1) = \la^2 z_1, \ga_2 (z_2) = \mu^2  z_2$.

Hence, setting $$  f (z_1, z_2) : = \sum_{i,j} a_{i,j} z_1^i  z_2^j ,$$ we get
 $$f (\la^2 z_1, \mu^2 z_2) = \la^{-n}  f (z_1, z_2) \Leftrightarrow
a_{i,j} \la^{2i} \mu ^{2j} -  \la^{-n} a_{i,j}=0 , \forall i,j \geq 0 .$$

Set now $j=0$: then  $a_{i,0} \la^{2i}  -  \la^{-2n} a_{i,0}=0$ holds for each $\la$, hence we get a  Laurent polynomial   in $\la$ whose coefficients are all vanishing,
hence $ a_{i,0} =0$ for each $ i \geq 0$. Therefore $f (z_1, 0) $ vanishes identically,
and $f (z_1, z_2)$ vanishes identically for $z_2=0$.

Varying now the maximal compact subgroup to which $\ga_1$ belongs, we obtain all the maximal compact subgroups
 to which $\ga_2$ belongs.

 Hence we have shown, for each choice of  $w_2$, that $f (z_1, z_2)$ vanishes identically for   $z_2= w_2$.

We conclude that the section determined by the
function $f$ vanishes identically on the surface $S$.

 \qed

\begin{remark}

One can formulate the last argument as showing that the Iitaka dimension of $L_j$ is   $- \infty$.

Note that a more general   result, but with a less elementary proof,  is contained in  Proposition IV.5.1 of \cite{mcquillan},
  saying that the canonical model of a surface foliation with numerical dimension $1$ has Iitaka dimension either   $- \infty$ or $1$
   (and in Example  II.2.3  it is stated that   the first alternative applies for   Hilbert modular surfaces).
\end{remark}

\section{Odd fake quadrics}
In this section, we assume that $S$ is an odd  fake quadric.
Recall that, over $\CC$,  $S$ contains no smooth rational curves \cite[Proposition~2.1.1]{miyaoka}
and in particular $K_S$ is ample.

\subsection{The intersection form}
\begin{lem}\label{lem:intersectionF1}
There exist $Q_1, Q_2 \in \Pic(S)$ such that
 $$Q_1^2=1, Q_2^2=-1, Q_1Q_2=0, K_S =  3Q_1-Q_2.$$
The numerical classes $[Q_1]$ and $[Q_2]$ are uniquely determined in $\mathrm{Num}(X)$.

Moreover, for any such $Q_1, Q_2$,
 \begin{enumerate}
    \item  $h^0(S, 3Q_1)\ge 1$ and $Q_1$ is nef and big;
    \item  $Q_1$ is ample unless $S$ contains an irreducible curve $C$ such that $C\sim Q_2$;
     \item   $Q_1$ is semiample.
   \end{enumerate}
\end{lem}

\begin{proof}
The intersection form on $\mathrm{Num}(S)$ is $\mathrm{diag}(1,-1)$.

Hence there exist divisors $Q_1, Q_2$ such that $Q_1^2=1, Q_2^2=-1, Q_1Q_2=0$.

 We may assume $K_S \cdot Q_1 \ge 0$ and $K_S \cdot Q_2\ge 0$ by possibly  replacing $Q_i$ with  $-Q_i$.
 Then $K_S\sim aQ_1+bQ_2$ with $a, b \in\ZZ, a\ge 0, b\le 0$.
 Since $K_S^2=8$, $a^2-b^2=8$. It follows  that $a=3, b=-1$ and $K_S \sim 3Q_1-Q_2$.

 Therefore $K_S=3Q_1-Q_2+\eta$ for some $\eta \in \mathrm{Tors}(S)$, and we can assume $K_S=3Q_1-Q_2$
 after replacing $Q_2$.

  Note that $h^2(S, 3Q_1)=h^0(S, K_S-3Q_1)=h^0(S,-Q_2)$ and $K_S(-Q_2)=-1$. Since $K_S$ is ample, $h^2(S, 3Q_1)=0$. Then the Riemann-Roch theorem shows
 $$h^0(S, 3Q_1)\ge \frac{1}{2}(3Q_1)(Q_2)+\chi(\hol_S)=1.$$

 Let $C$ be an irreducible curve. We write  $C\sim a Q_1+bQ_2$ with $a=CQ_1$ and $b=-CQ_2$.
 Then $K_SC=3a+b$ and $C^2=a^2-b^2$.

 In order to see whether  $Q_1$ is nef, respectively ample, assume that $CQ_1\le 0$, i.e., $a\le 0$.

 Then $C$ is a negative curve, and, by the next Proposition \ref{prop:noratcurve},
 the class of $C$ equals the class of $(b-1)Q_1+bQ_2$, and we are done unless $b \leq 1$.

 However, since $K_S$ is ample and $3(b-1) + b = K_SC>0$, $b \geq 1$. If $b=1$, then $C \sim  Q_2$.
 Then we have shown that $Q_1$ is nef, and then (1) and  (2) are  proven.

For (3), we may assume that $Q_1$ is not ample. Then by (2), there is an irreducible curve $C\sim Q_2$. Note that $p_a(C)=1$ and thus $\mathcal O_C(K_S+C)\cong \omega_C \cong  \hol_C$. Moreover,
 $$3Q_1=K_S+C+\eta'$$
 for some $\eta' \in \mathrm{Tors}(S)$.

There exists $m>0$ such that
\begin{itemize}
    \item $m\eta'=0$, and thus $3mQ_1|_C\cong \hol_C$.
    \item $h^0(3mQ_1)\gg 0$.
\end{itemize}
Note that $Q_1$ is nef and big, and that
$$3mQ_1-C\sim K_S+3(m-1)Q_1.$$
By Kawamata-Viehweg vanishing theorem, we have
$$H^1(S, 3mQ_1-C)=0.$$
So the trace of $|3mQ_1|$ on $C$ is complete and base-point-free.

Write $|3mQ_1|=|M|+F$, where $|M|$ is the movable part and $F$ is the fixed part. The discussion above shows that $F\not\ge C$ and thus $FC\ge 0$. Since $3mQ_1.C=0$, we conclude that $M.C=0$ and $F.C=0$.
It follows that $M\sim \lambda Q_1$ for some positive integer $\lambda$.
Because $\mathrm{Tors}(S)$ is finite, $|kQ_1|$ has no fixed part for sufficiently large and divisible $k>0$. By a theorem of Zariski, $Q_1$ is semiample.
\end{proof}

Unlike the even fake quadric case, we do not know whether $S$ contains a negative curve or not  (but in the case it does not contain such a negative curve, we have determined the condition that it is a Mori deam space in Theorem \ref{thm:moridream}).

\begin{prop}\label{prop:noratcurve}
Let $C$ be an irreducible curve on $S$. Assume that $C^2<0$. Then
\begin{enumerate}
    \item $C\sim aQ_1+(a+1)Q_2$ for  some $a \in \ZZ_{\ge 0}$ and $p_a(C)=a+1$.
    \item For any irreducible curve $C_0 \not=C$, $C_0^2\ge 0$.
    \item Set $D:=(a+1)Q_1+aQ_2$. Then $DC=0$, $D$ is nef and big,
    moreover    $D$ is semiample  only if $\hol_C(D)$ is a torsion divisor.
    \item One of the sides of $\mathrm{Eff}(S)$ is $\RR_+[C]$ and one of the sides of $\mathrm{Nef}(S)$ is $\RR_+[D]$.
\end{enumerate}
\end{prop}

\begin{proof}We may assume $C\sim aQ_1+bQ_2$ hence $a=CQ_1$ and $b=-CQ_2$.
Then by our assumption  $K_S   C=3a+b >0,  C^2=a^2-b^2 < 0$.

Set $ \al : = |a| $, then $ |b| = \al + \de$ with $\de >0$.

For some  $\e_1, \e_2 \in \{ 1, -1 \}$, we have $ a = \e_1 \al, b = \e_2 (\al + \de).$

The inequalities   $K_S   C=3a+b >0$ and $K_S   C + C^2 \geq 0$ (since $S$ contains no rational curve)
read out as:
$$ 3 \e_1 \al +  \e_2 (\al + \de) > 0, \ \  3 \e_1 \al +  \e_2 (\al + \de)  - 2 \de \al - \de^2 \geq 0.$$

If $\de \geq 2$, then $ - 2 \de \al  +  (3 \e_1  +  \e_2 )  \al \leq 0$, while $  \e_2  \de -  \de^2  <0$; hence
this contradiction shows that $\de=1$.

The first inequality excludes the possibility $\e_1 = \e_2 = -1$.

If $\e_1 = 1, \e_2 = -1$, then the second inequality tells that $ - 2 \geq 0$, absurd.

Hence $\e_2=1$, and $ (3 \e_1  + 1    - 2 ) \al  \geq 0$ shows that $\e_1 =1$.

Therefore $C\sim \al Q_1+ (\al + 1) Q_2$ and (1) is proven.

(2) follows right away because if $C_0$ is a different  negative curve,
$C_0\sim a_0Q_1+(a_0+1)Q_2$ with $a_0\ge 0$.

Then $C   C_0=aa_0-(a+1)(a_0+1)<0$, this is impossible and (2) is proven.

(3): clearly we  have $D^2=2a+1>0$ and $DC=0$.

From (1) and (2), we see $DC_0>0$ for any $C_0 \not= C$. In particular, $D$ is nef and big.

 By a Theorem of Zariski, saying that a nef and big divisor $D$
is asymptotically base point free if and only if there exists a large multiple $|mD|$ which  is without fixed part,
$D$ is semiample if and only if for each irreducible curve $C'$, $C'$ is not in the base locus of some $|mD|$
with $m$ positive.

Applying this to $C' = C$ we see that $\hol_C(D)$ must be  a torsion line bundle.
\end{proof}

\begin{cor}\label{lem:fakeF1fib}Assume that $f\colon S\rightarrow \PP^1$ is a fibration with
general fibre $F$. Then either $F\sim a(Q_1+Q_2)$ with $a \ge 1$ and $g(F)=2a+1$,  or $F\sim a(Q_1-Q_2)$ and $g(F)=a+1$ with $a\ge 2$.
\end{cor}

\begin{proof}
Let $F \sim a Q_1 + b Q_2$: then $F^2=0$ amounts to $ a^2 = b^2$, that is,
$ a= \e_1 \al, b = \e_2 \al$, with $\al >0$,  $\e_j \in \{1, -1\}$.

 Since $K_S   F >0$, we get $ 3\e_1  + \e_2 >0$, hence $\e_1 =1$, and the two solutions are as stated.
\end{proof}

\begin{cor}\label{lem:fakeF1negaonefib}
Assume that $S$ contains a negative curve. Then $S$ admits at most one fibration, and  if there is a fibration $f \colon S \rightarrow \PP^1$  with  general fibre $F$, then $F\sim (g(F)-1)(Q_1-Q_2).$
\end{cor}
\proof
 To show the last assertion, we observe that $F = Q_1 + \e Q_2$ is nef, hence $F \cdot C \geq 0$.

This condition amounts to $ a - \e (a+1) \geq 0$, hence $\e=-1$.

\qed

\section{Characteristic $p$}

We begin with an easy remark: let $S$ be an even  fake quadric (over an algebraically closed field of characteristic $p >0$).

Then the quadrant $ \{ n_1 L_1 + n_2 L_2 | n_1, n_2 \geq  0\}$ is contained in the closure of the effective cone,
since $\sP : =  \{ n_1 L_1 + n_2 L_2 | n_1, n_2 >  0\}$ consists of big divisors $D$ (this means, for $ n \gg 0$ ,
$ nD = A + E$, where $A$ is ample and $E$ effective.

Indeed, by Riemann Roch $D \sim  d_1 L_1 + d_2 L_2$ is effective for $ d_1, d_2 \geq 2$, $(d_1, d_2) \neq (2,2)$.

\begin{lemma}\label{negative}
There are   at most two   negative irreducible curves $C$ on $S$.

 The class of  $C$ may only be $C\sim - L_1+b L_2$ or $C\sim a L_1- L_2$ and $C \cong \PP^1$.
Moreover, $ b \geq 2$, and $a \geq 2$ if $K_S$ is ample.

 If $K_S$ is not ample there  is a unique irreducible
$-2$-curve $C$  orthogonal to $K_S$: then either  $C\sim - L_1+ L_2$ or $C\sim  L_1- L_2$,
but obviously both possibilities cannot occur.

\end{lemma}
\begin{proof}
If $C$ is irreducible with $ C \sim  c_1 L_1 + c_2 L_2$, if $C$ is negative $ c_1 c_2 < 0$,
hence we may assume that $c_1 > 0, c_2  < 0$.

Since
$ K_S   C \geq 0$, we obtain  $ c_1 + c_2 \geq 0$.

If $C'$ is another negative irreducible curve , it cannot lie in the same quadrant, since
$c'_1 > 0, c'_2  < 0$ implies $C   C' = c_1 c'_2  + c'_1c_2 <0$, a contradiction.

Hence there is at most one negative curve, in each of the two quadrants which are neither positive nor negative.

 Assume now that we have an  irreducible curve $C$ with $C^2<0$. Then  $C\sim a L_1+b L_2$ for $a, b\in \ZZ$ and
$$C^2=2ab,~~ K_S   C=2a+2b.$$
We may assume that $a >0$ and $b<0$. By the adjunction formula,
$$-2\le 2p_a(C)-2=C^2+K_SC=2(a+b+ab)=  2a (1+b) + 2 b.$$
Therefore $b=-1$, $C \cong \PP^1$.

\end{proof}

\begin{remark}
i) In the paper \cite{EST} Ekedahl, Shepherd Barron and Taylor show that for each prime $p$
which is inert in the quadratic field $\sK$ which is the centre of the quaternion algebra $\sA$
of a Shavel like surface, then the divisors $ -2 L_1+2p   L_2$ and  $ 2p   L_1 -2 L_2$
are effective.

ii) is it true that the possible numbers $a,b$ in the previous Lemma \ref{negative} can only be equal to
the characteristic $p$?

\end{remark}

If we have two  negative curves $C_1 \sim a_1 L_1- L_2$, $C_2 \sim - L_1+ b_2 L_2$,
then these two curves span the Effective cone, which is therefore polyhedral.

The Nef cone consists of divisors $D  \sim a  L_1 + b  L_2$ such that
$$ a \leq b a_1, b \leq a b_2,$$
hence it is polyhedral and spanned by
$D_1   \sim a_1  L_1 +   L_2$, $D_2  \sim   L_1 + b_2  L_2$.

\begin{prop}\label{mds}
If on $S$ there are  two  negative curves $C_1 \sim a_1 L_1- L_2$, $C_2 \sim - L_1+ b_2 L_2$,
then $S$ is a Mori Dream Space.

\end{prop}
\begin{proof}
By \cite{K3} it suffices to show that the two divisors $D_1   \sim a_1  L_1 +   L_2$, $D_2  \sim   L_1 + b_2  L_2$
are semiample.

The divisors are both nef and big, and
by symmetry, it suffices to show only the first assertion, that $D_1$ is semiample.

We denote by $\mathbb{E} (D_1)$ the exceptional locus of $D_1$, i.e. the union of the finite maximal subvarieties $Z$ such that the restriction of $D_1$
to $Z$ is non big. Since $D_1$ is big and $C_1$ is the only curve which is orthogonal to $D_1$, it follows that $\mathbb E (D_1)=C_1$.
By Lemma \ref{negative}, we have that $C_1 \cong \PP^1$ and hence $ \hol_{C_1}(D_1)$ is semiample.

We apply Theorem 0.2 of \cite{keel} (see also \cite{cmm} Corollary 3.6), stating that if we are
in positive characteristic and $D_1$ is nef and big and the restriction to the exceptional locus
$\mathbb{E} (D_1)$ is semiample, then also $D_1$ is semiample.

Hence we are done.

\end{proof}

Theorem \ref{shavel} follows now immediately from Theorem \ref{thm:shavel} and the previous Proposition
\ref{mds}.

\section{Problems}

\begin{itemize}
\item
Problem~1:~ Assume that $S$ is isogenous to a product of curves, of unmixed type: is $S$ an even  fake quadric?

        Determine more generally when  a fake quadric $S$, isogenous to a product,  is an odd  fake quadric.
    \item Problem~2:~Let $S$ be an odd  fake quadric. Does  $S$ contain a negative curve?
    \item Problem~3:~Let $S$ be an odd  fake quadric. Could $S$ have two fibrations?

          Remark:~ this is related to Problem 1 since   surfaces isogenous to a product are $\QQ$-homology quadrics
           having two fibrations.

    \item Problem~4'~(Hirzebruch' s  question):
Is every surface homeomorphic to a smooth quadric indeed a deformation of $\PP^1 \times \PP^1$?

   \item Problem~4'':
Is every surface homeomorphic to $\FF_1$ indeed a deformation of $\FF_1$?

A negative answer to both questions would follow if one could answer positively the next question 5, or  negatively the
weaker question 6: indeed, by Michael Freedman's Theorem \cite{freedman} a simply connected fake quadric is homeomorphic
either to  $\FF_1$ or to $\FF_0 = \PP^1 \times \PP^1$. 
\item
 Problem~5:~Let $S$ be a fake quadric: is then the universal covering of $S$ biholomorphic to $\HH \times\HH$?
 \item
 Problem~6:~Is there a simply connected   fake quadric?
 \item
 Problem~7: (raised by Michael L\"onne): is  there a fake quadric with $H_1(S, \ZZ) =0$?

\end{itemize}

\begin{remark}
If a fake quadric $S$  is homeomorphic to $\PP^1 \times \PP^1$ then $S$ is spin, that is, $K_S$ is divisible by $2$,
and one may study its half-canonical ring.
\end{remark}

\end{document}